\documentclass[12pt,reqno]{amsart}

\allowdisplaybreaks[2]

\usepackage[all,poly]{xy}
\usepackage{amsfonts}
\usepackage{eufrak}
\usepackage{amssymb}
\usepackage{amsmath}
\usepackage{mathrsfs}
\usepackage{color}
\usepackage[pagebackref,colorlinks]{hyperref}
\usepackage{enumerate}

\usepackage[pagebackref,colorlinks]{hyperref}

\usepackage{comment}
\numberwithin{equation}{section}

\theoremstyle{plain}
\newtheorem{thm}{Theorem}[section]

\newtheorem{cor}[thm]{Corollary}
\newtheorem{lemma}[thm]{Lemma}

\theoremstyle{definition}
\newtheorem{deff}[thm]{Definition}
\newtheorem{example}[thm]{Example}
\theoremstyle{remark}
\newtheorem{rmk}[thm]{\bf Remark}

\def\mG{\mathcal{G}}

\def\N{\mathbb N}

\def\r{\rangle}
\def\O{\mathcal{O}}

\def \Z{\mathbb Z}
\def\-{\text{-}}
\def\TT{\mathcal{T}}
\def\LL{\mathcal{L}}
\def\FF{\mathcal{F}}

\newcommand{\End}{\operatorname{End}}

\newcommand{\ann}{\operatorname{ann}}

\newcommand{\op}{\operatorname{op}}

\begin{document}

\title{$\Z$-graded rings as Cuntz--Pimsner rings}

\author {Lisa Orloff Clark}
\address {Lisa Orloff Clark\\School of Mathematics and Statistics, Victoria University of Wellington, PO Box 600, Wellington 6140, New Zealand}
\email {lisa.clark@vuw.ac.nz}

\author {James Fletcher}
\address {James Fletcher\\School of Mathematics and Statistics, Victoria University of Wellington, PO Box 600, Wellington 6140, New Zealand}
\email {james.fletcher@vuw.ac.nz}

\author{Roozbeh Hazrat}
\address{Roozbeh Hazrat\\
Centre for Research in Mathematics, Western Sydney University, Australia}
\email{r.hazrat@westernsydney.edu.au}

\author{Huanhuan Li}
\address{Huanhuan Li\\
Centre for Research in Mathematics, Western Sydney University, Australia}
\email{h.li@westernsydney.edu.au}

\thanks {Clark and Fletcher were supported by Marsden grant 15-UOO-071 from the Royal Society of New Zealand. Hazrat and Li acknowledge Australian Research Council grant DP160101481.}

\subjclass[2010]{}


\date{\today}

\begin{abstract} 
Given a $\Z$-graded ring $A$ and a subring $R\subseteq A$, it is natural to ask whether $A$ can be realised as the Cuntz--Pimsner ring of some $R$-system. In this paper, we derive sufficient conditions on $A$ and $R$ for this to be the case. As an application, we give conditions under which the Steinberg algebra $A_K(\mathcal{G})$ associated to a $\Z$-graded groupoid $\mathcal G=\sqcup_{n\in \Z} \mG_n$ can be realised as the Cuntz--Pimsner ring of an $A_K(\mathcal{G}_0)$-system.
\end{abstract}

\maketitle

\section{Introduction}

Cuntz--Pimsner rings were introduced by Carlsen and Ortega in \cite{CO11} as an algebraic analogue of Cuntz--Pimsner $C^{\ast}$-algebras (first introduced by Pimsner \cite{Pimsner} and later refined by Katsura \cite{Katsura}).  To construct a Cuntz--Pimsner algebra one requires a $C^*$-algebra $A$ (often called the coefficient algebra) along with a $C^*$-correspondence $X$ over $A$ (loosely speaking, this can be thought of as a Hilbert space where the inner product takes values in a $A$ rather than just the complex numbers). One then constructs the Toeplitz algebra $\mathcal{T}_X$ associated to the correspondence, and defines the Cuntz--Pimsner algebra $\mathcal{O}_X$ as a suitable quotient. The Cuntz--Pimsner algebra $\mathcal{O}_X$ contains a faithful copy of $A$ and naturally carries a gauge action of the circle group (which can in turn be used to define a $\Z$-grading). Of key importance to the the study of Toeplitz and Cuntz--Pimsner algebras are numerous results that relate structural properties of $\mathcal{T}_X$ and $\mathcal{O}_X$ to structural properties of the coefficient algebra $A$ (for example, \cite[Theorem~7.1]{Katsura} deals with exactness, \cite[Theorem~7.2 and Corollary~7.4]{Katsura} with nuclearity, \cite[Theorem~8.6 and Proposition~8.8]{Katsura2} with the gauge-invariant ideal structure, and \cite[Proposition~8.2 and Theorem~8.6]{Katsura} with operator $K$-theory). This of course serves as motivation for attempting to realise $C^*$-algebras as Cuntz--Pimsner algebras: given a $C^*$-algebra $B$, can we find a subalgebra $A\subseteq B$ (which we hopefully already know some information about) and a $C^*$-correspondence $X$ over $A$ such that $B$ is isomorphic to $\mathcal{O}_X$? Various $C^*$-algebraic constructions are well-known to be realisable (nontrivially) in terms of Cuntz--Pimsner algebras: for example crossed products by automorphisms \cite[Example~3]{Pimsner}, graph $C^*$-algebras \cite[Example~8.13]{Raeburn}, and more generally topological graph algebras \cite{Katsura3}. 

Similarly, to construct a Toeplitz (or Cuntz--Pimsner) ring one requires a coefficient ring $R$, two $R$-bimodules $P$ and $Q$, and an $R$-bimodule homomorphism $\psi:P\otimes_R Q\rightarrow R$.  Such data is called an $R$-system. One can think of the bimodules $P$ and $Q$ as playing the role of the $C^*$-correspondence, whilst the map $\psi$ plays the role of an $R$-valued `inner product'. The associated Toeplitz ring $\mathcal{T}_{(P,Q,\psi)}$ is then defined in much the same way as in the $C^*$-algebraic setting. The Toeplitz ring then naturally carries a $\Z$-grading. Unfortunately, when moving to the purely algebraic setting, the missing analytic structure forces the authors of \cite{CO11} to impose additional technical hypotheses in order to pass from $\mathcal{T}_{(P,Q,\psi)}$ to the Cuntz--Pimsner ring $\mathcal{O}_{(P,Q,\psi)}$ (see \cite[Proposition~3.11 and Remark~4.10]{CO11}). We summarise all the necessary background material that we will require in Section~\ref{preliminaries}.  Despite these difficulties, a number of positive results have still been derived for Cuntz--Pimsner rings. For example, \cite[Corollary~5.4]{CO11} gives necessary and sufficient conditions for a graded representation of $\mathcal{O}_{(P,Q,\psi)}$ to be injective, \cite[Corollary~7.29]{CO11} gives a complete description of the lattice of two sided graded ideals of $\mathcal{O}_{(P,Q,\psi)}$, whilst \cite[Theorem~7.3]{COP12} gives necessary and sufficient condition for simplicity.  Of course, this construction, along with the positive results, would be redundant without a collection of interesting examples of rings that can be realised (nontrivially) as Cuntz--Pimsner rings. Carlsen and Ortega present three examples to motivate their construction:
\begin{enumerate}
\item[(i)] 
The crossed product of a ring $R$ by an automorphism can be realised as the Cuntz--Pimsner ring of an $R$-system \cite[Example~5.5]{CO11};
\item[(ii)]
Given a unital ring $R$, an idempotent $p\in R$, and a ring isomorphism $\phi:R\rightarrow pRp$, the fractional skew monoid ring $R[t_+,t_-,\alpha]$ (as constructed in  \cite{ara2004}) can be realised as the Cuntz--Pimsner ring of an $R$-system \cite[Example~5.7]{CO11};
\item[(iii)]
Given a directed graph $E$ and a unital commutative ring $K$, if we let $\oplus_{v\in E^0}K$ denote the ring of finitely supported $K$-valued functions on the vertex set $E^0$, then the Leavitt path algebra $L_K(E)$ can be realised as the Cuntz--Pimsner ring of an $\oplus_{v\in E^0}K$-system \cite[Examples~1.10 and 5.8]{CO11}.
\end{enumerate}
In \cite{CO11} these three examples are treated individually using various ad hoc techniques, and thus one of the main goals of this paper is to provide a systematic approach for dealing with examples of this sort. 

When we began this project, our original goal was to `algebraify' the results of \cite{RRS17} by investigating how the Steinberg algebra $A_K(\mathcal{G})$ associated to a groupoid $\mathcal{G}$ equipped with a continuous $\Z$-valued cocycle $c$ fits into the Cuntz--Pimsner framework. In particular, we wanted to show that $A_K(\mathcal{G})$ can be realised as the Cuntz--Pimsner ring of an $A_K(c^{-1}(0))$-system. Eventually we realised that our arguments could be applied in more general settings, and that is what we report in this paper. Our main result gives a construction for recognizing and building Cuntz--Pimsner rings:  given a ring $R$ and subring $A\subseteq R$, Theorem~\ref{mainth0} tells us what additional conditions are required in order for there to exist a graded isomorphism from $R$ to the Cuntz--Pimsner ring of some $A$-system. Our construction is very general and does not impose a lot of technical hypotheses.  At the same time, we show how the the three key examples listed above can be recovered (relatively easily) using our procedure. Finally, in Section~\ref{Steinberg algebras} we apply our results to Steinberg algebras associated to $\Z$-graded groupoids and complete our original goal of proving an algebraic analogue of \cite[Proposition~10]{RRS17}.

\section{Preliminaries: Cuntz--Pimsner rings}
\label{preliminaries}

In this section, we present some preliminary results about graded rings and recall from \cite{CO11} the basic construction of Cuntz--Pimsner rings.   A ring $A$ (possibly without unit) is called a \emph{$\Z$-graded ring} if $ A=\bigoplus_{ n \in \Z} A_{n}$
with each $A_{n}$ an additive subgroup of $A$ and $A_{n} A_{m}\subseteq A_{n+m}$ for all $n,m \in \Z$. Elements of $\cup_{n\in \Z}A_n$ are called \emph{homogeneous elements} of $A$, and a nonzero element $a\in A_n$ is said to have \emph{degree} $n$, which we denote by $|a|=n$. If $A$ is an algebra over a ring $K$, then $A$ is called a \emph{graded algebra} if $A$ is a graded ring and $A_{n}$ is a $K$-submodule for each $n \in \Z$. For the basics on graded ring theory see \cite{hazi,no}.

Let $R$ be a ring. Given two $R$-bimodules $P$ and $Q$ we denote by $P\otimes_R Q$ the $R$-balanced tensor product. An \emph{$R$-system} is a triple $(P,Q,\psi)$ where $P$ and $Q$ are $R$-bimodules, and $\psi:P\otimes_R Q\rightarrow R$ is an $R$-bimodule homomorphism; see \cite[Definition~1.1]{CO11}. 

Given an $R$-system $(P,Q,\psi)$, we say that a quadruple $(S, T, \sigma, B)$ is a \emph{covariant representation} of $(P,Q,\psi)$ on $B$ if:
\begin{itemize}
\item[(i)] $B$ is a ring;

\item[(ii)] $S : P\rightarrow{} B$ and $T: Q \rightarrow{} B$ are additive maps;

\item[(iii)] $\sigma: R\rightarrow{} B$ is a ring homomorphism;

\item[(iv)] $S(pr) = S(p)\sigma(r), S(rp) = \sigma(r)S(p), T(qr) = T(q)\sigma(r)$ and $T(rq) = \sigma(r)T(q)$ for $r\in R, p\in P$ and $q\in Q$;

\item[(v)] $\sigma(\psi(p\otimes_R q)) = S(p)T(q)$ for $p\in P$ and $q\in Q$.
\end{itemize} We denote by $R\langle S, T, \sigma\rangle$ the subring of $B$ generated by $\sigma(R)\cup T(Q) \cup S(P)$. If $R\langle S, T, \sigma\rangle=B$, then we say that the covariant representation $(S, T, \sigma,B)$ is \emph{surjective}, and if the ring homomorphism $\sigma$ is injective, then we say that the covariant representation $(S, T, \sigma,B)$ is \emph{injective}.

Given an $R$-bimodule $P$, we define $P^{\otimes 0}:=R$, $P^{\otimes 1}:=P$, and for $n\geq 2$, we define $P^{\otimes n}$ recursively by $P^{\otimes n}=P\otimes_R P^{\otimes (n-1)}$. Each $P^{\otimes n}$ then naturally has the structure of an $R$-bimodule. If $P$ and $Q$ are $R$-bimodules and $\psi:P\otimes_R Q\rightarrow R$ is an $R$-bimodule homomorphism, we define $\psi^{\otimes 0}:P^{\otimes 0}\otimes_R Q^{\otimes 0}\rightarrow R$ by $\psi^{\otimes 0}(r\otimes_r s):=rs$ for $r,s\in R$, $\psi^{\otimes 1}:=\psi$, and for $n\geq 2$, we define $\psi^{\otimes n}:P^{\otimes n}\otimes_R Q^{\otimes n}\rightarrow R$ recursively by 
\begin{align*}
\psi^{\otimes n}\big((p\otimes_R p')\otimes_R (q\otimes_R q')\big)&:=\psi\big(p\cdot\psi^{\otimes (n-1)}(p'\otimes_R q)\otimes_R  q'\big)\\
&\quad \text{for $p\in P$, $p'\in P^{\otimes (n-1)}$, $q\in Q^{\otimes (n-1)}$, $q'\in Q$.}
\end{align*}
It is routine to check that each $\psi^{\otimes n}$ is an $R$-bimodule homomorphism. Similarly, if $P$ is an $R$-bimodule, $B$ is a ring, and $S:P\rightarrow B$ is an additive map, we define $S^{\otimes 1}:=S$, and for $n\geq 2$, we define an additive map $S^{\otimes n}:P^{\otimes n}\rightarrow B$ recursively by $S^{\otimes n}(p\otimes p'):=S(p)S^{\otimes (n-1)}(p')$ for $p\in P$, $p'\in P^{\otimes (n-1)}$. It is not difficult to show that if $(S,T,\sigma,B)$ is a covariant representation of an $R$-system $(P,Q,\psi)$, then for each $n\geq 0$, $(S^{\otimes n},T^{\otimes n},\sigma,B)$ (with $S^{\otimes 0}:=T^{\otimes 0}:=\sigma$) is a covariant representation of the $R$-system $(P^{\otimes n},Q^{\otimes n},\psi^{\otimes n})$ (see \cite[Lemma~1.5]{CO11}). It follows that
\[
R\langle S, T, \sigma\rangle
=\mathrm{span}\{T^{\otimes m}(q)S^{\otimes n}(p):m,n\geq 0, q\in Q^{\otimes m}, p\in P^{\otimes n}\}. 
\]
Furthermore, it follows that $R\langle S, T, \sigma\rangle$ is $\Z$-graded with 
\[
R\langle S, T, \sigma\rangle_t
=\mathrm{span}\{T^{\otimes m}(q)S^{\otimes n}(p):m,n\geq 0, m-n=t, q\in Q^{\otimes m}, p\in P^{\otimes n}\}
\]
for each $t\in \Z$ (see \cite[Proposition 3.1]{CO11}). 

By \cite[Theorem 1.7]{CO11}, for an $R$-system $(P,Q,\psi)$, there exists an injective and surjective covariant representation $$(\iota_P, \iota_Q, \iota_R, \TT_{(P,Q,\psi)})$$ with the following property:   if $(S, T, \sigma,B)$ is a covariant representation of $(P,Q,\psi)$, then there exists a unique ring homomorphism $$\eta_{(S,T,\sigma,B)}:\TT_{(P,Q,\psi)}\rightarrow{}B$$ such that $\eta_{(S,T,\sigma,B)}\circ \iota_R =\sigma$, $\eta_{(S,T,\sigma,B)}\circ \iota_Q =T$, and $\eta_{(S,T,\sigma,B)}\circ \iota_P=S$. We call $(\iota_P, \iota_Q, \iota_R, \TT_{(P,Q,\psi)})$ the \emph{Toeplitz representation} of $(P,Q,\psi)$, and $\TT_{(P,Q,\psi)}$ the \emph{Toeplitz ring} of $(P,Q,\psi)$.

Let $(P,Q,\psi)$ be an $R$-system. Then a right $R$-module homomorphism $t : Q_R\rightarrow{} Q_R$ is called \emph{adjointable} with respect to $\psi$  if there exists a left $R$-module homomorphism $s:{ }_RP\rightarrow \, { }_RP$ such that 
\[\psi(p\otimes_R t(q))=\psi(s(p)\otimes_R q), \qquad \text{for $p\in P$, $q\in Q$.}\] 

We call $s$ an \emph{adjoint} of $t$ with respect to $\psi$. We write $\LL_P (Q)$ for the set of all adjointable
homomorphisms (with respect to $\psi$). We denote by $\LL_Q(P)$ the set of all the adjoints.

Given an $R$-system $(P,Q,\psi)$, for each $p\in P$ and $q\in Q$, we define homomorphisms $\theta_{q, p}: Q_R\rightarrow{}Q_R$ and $\theta_{p, q}: {}_RP\rightarrow {}_RP$ by 
\[
\theta_{q,p}(x):=q\psi(p\otimes_R x)
\quad \text{and} \quad
\theta_{p, q}(y):=\psi(y\otimes_R q)p
\quad \text{for each $x\in Q$ and $y\in P$.}
\]
Then $\theta_{q,p}\in \LL_P (Q)$ with $\theta_{p,q}$ as an adjoint. We call the homomorphisms $\{\theta_{q, p}:q\in Q,p\in P\}$ \emph{rank-one adjointable homomorphisms}, and we denote the linear span of all such homomorphisms by $\FF_P (Q)$. Similarly, we denote by $\FF_Q(P)$ the set of all rank-one adjoints.

We say that an $R$-system $(P,Q,\psi)$ satisfies \emph{condition} (FS) if for all finite sets $\{q_1, \cdots, q_n\}\subseteq Q$ and $\{p_1,\cdots, p_m\}\subseteq P$ there exist $\Theta\in\FF_P (Q)$ and $\Phi\in \FF_Q(P)$ such that $\Theta(q_i) = q_i$ and $\Phi(p_j)= p_j$ for every $i = 1,\cdots, n$ and $j = 1,\cdots, m$, respectively. By \cite[Lemma~3.8]{CO11}, if $(P,Q,\psi)$ satisfies \emph{condition} (FS) then so does $(P^{\otimes n},Q^{\otimes n},\psi^{\otimes n})$ for each $n\geq 1$.

Let $R$ be a ring and $(P,Q,\psi)$ an $R$-system. We define ring homomorphisms $\Delta: R\rightarrow {} \End_R(Q_R)$ and $\Gamma: R\rightarrow{}\End_R(_RP)^{\op}$ by
\begin{equation}
\label{delta}
\Delta(r)(q):=rq
\quad \text{and} \quad
\Gamma(r)(p):=pr
\quad \text{for $r\in R$, $p\in P$, $q\in Q$.}
\end{equation}
Note that for every $r\in R$ we have that $\Gamma(r)$ is the adjoint of $\Delta(r)$, and thus $\Delta(r)\in \LL_P (Q)$ and $\Gamma(r)\in \LL_Q(P)$.

\begin{lemma}\label{lemmaforpi} $($\cite[Lemma 2.2]{KPW98} and \cite[Proposition 3.11]{CO11}$)$ Let $R$ be a ring and $(P,Q,\psi)$ an $R$-system satisfying condition \rm{(FS)} and let $(S, T, \sigma,B)$ be a covariant representation of $(P,Q,\psi)$. Then
there exists a unique ring homomorphism $\pi_{T,S}: \FF_P(Q)\rightarrow{}B$ such that 
\[
\pi_{T,S}(\theta_{q,p}) =T(q)S(p) \quad \text{for $p\in P$, $q\in Q$.}
\]
\end{lemma}

We say that a two-sided ideal $J$ of $R$ is $\psi$-\emph{compatible} if $J\subseteq \Delta^{-1}(\FF_P (Q))$, and we say
that a $\psi$-compatible two-sided ideal $J$ of $R$ is \emph{faithful} if $J\cap \ker\Delta= \{0\}$. For a $\psi$-compatible two-sided ideal $J$ of $R$, we define $\TT(J)$ to be the two-sided
ideal of $\TT_{(P,Q,\psi)}$ generated by $\{\iota_R(x)-\pi_{\iota_Q,\iota_P}(\Delta(x)): x \in J\}$.

The \emph{Cuntz--Pimsner ring} relative to the $\psi$-compatible ideal $J$ is the quotient ring 
\[
\O_{(P,Q,\psi)}(J)=\TT_{(P,Q,\psi)}/\TT(J).
\]
Note that $\O_{(P,Q,\psi)}(J)$ is a $\Z$-graded algebra whose grading is inherited from $\TT_{(P,Q,\psi)}$. We denote by $\rho_J$ the quotient map $\rho_J:\TT_{(P,Q,\psi)}\rightarrow{}\O_{(P,Q,\psi)}(J)$. A covariant representation $(S, T, \sigma,B)$ of $(P,Q,\psi)$ is said to be \emph{Cuntz--Pimsner invariant} relative to $J$ if $\pi_{T,S}(\Delta(x)) = \sigma(x)$ for every $x\in J$. If we let $\iota^J_R:= \rho_J \circ \iota _R, \iota^J_Q:=\rho_J\circ \iota_Q$ and $\iota^J_P:=\rho_J\circ \iota_P$, then $(\iota^J_P, \iota^J_Q, \iota^J_R,\O_{(P,Q,\psi)}(J))$ is a surjective covariant representation of $(P,Q,\psi)$ which is Cuntz--Pimsner invariant relative to $J$ (see \cite[Theorem 3.18]{CO11}). It follows that $\O_{(P,Q,\psi)}(J)$ also has a universal property: if $(S, T, \sigma,B)$ is a covariant representation of $(P,Q,\psi)$ that is Cuntz--Pimsner invariant relative to $J$,  then there exists a unique ring homomorphism 
\[
\eta^J_{(S,T,\sigma,B)}:\O_{(P,Q,\psi)}(J)\rightarrow{}B
\]
such that $\eta^J_{(S,T,\sigma,B)}\circ \iota_R^J =\sigma$, $\eta^J_{(S,T,\sigma,B)}\circ \iota_Q^J =T$, and $\eta^J_{(S,T,\sigma,B)}\circ \iota_P^J=S$.

Given a faithful $\psi$-compatible two-sided ideal $J$ of $R$, we say that $J$ is \emph{maximal} if $J=J'$ whenever $J'$ is a faithful $\psi$-compatible two-sided ideal of $R$ with $J\subseteq J'$. We say that $J$ is \emph{uniquely maximal} if $J'\subseteq J$ for every faithful $\psi$-compatible two-sided ideal $J'$ of $R$.

\begin{deff} Let $R$ be a ring and $(P,Q,\psi)$ an $R$-system satisfying condition (FS). If there exists a uniquely maximal faithful
$\psi$-compatible two-sided ideal $J$ of $R$, then we define the Cuntz--Pimsner ring of $(P,Q,\psi)$ to be the ring $$\O_{(P, Q, \psi)}:=\O_{(P, Q, \psi)}(J)$$ and we let $(\iota^{CP}_P, \iota^{CP}_Q, \iota_{R}^{CP},\O_{(P, Q, \psi)})$ denote the covariant representation $$(\iota^{J}_P, \iota^{J}_Q, \iota_{R}^{J},\O_{(P, Q, \psi)}(J))$$ and call it the \emph{Cuntz--Pimsner representation} of $(P,Q,\psi)$. 
\end{deff}

Let $J$ be a two-sided ideal of a ring $R$. We let $J^{\perp}$ denote the two-sided ideal $\{r\in R : ry = yr = 0 \text{ for all } y\in J\}$.

\begin{lemma}\label{lemmaforexistence}\cite[Lemma 5.2]{CO11} Let $R$ be a ring and let $(P,Q,\psi)$ be an $R$-system which satisfies condition 
\rm{(FS)}. If $\Delta^{-1}(\FF_P (Q)\cap (\ker \Delta)^{\perp}\cap \ker\Delta= \{0\}$, then $J: = \Delta^{-1}(\FF_P (Q))\cap (\ker \Delta)^{\perp}$ is a uniquely
maximal faithful $\psi$-compatible two-sided ideal of $R$. Thus the Cuntz--Pimsner ring of $(P,Q,\psi)$ is defined in this case.
\end{lemma}

We finish this section with one of the properties of Cuntz--Pimsner rings which will be used in later sections.

\begin{thm} \label{gut} $($\cite[Corollary 5.4]{CO11} and \cite[Theorem 6.4]{Katsura}, The Graded Uniqueness Theorem$)$ Let $R$ be a ring and let $(P,Q,\psi)$ be an $R$-system which satisfies condition \rm{(FS)}, and assume that there exists a uniquely maximal faithful $\psi$-compatible two-sided ideal of $R$. If $A$ is a $\mathbb Z$-graded ring and $\eta : \O_{(P,Q,\psi)} \rightarrow A$ is a graded ring homomorphism with $\eta (\iota_R^{CP}(r)) \not =  0$ for every $r \in R \backslash \{0\}$, then $\eta$ is injective.
\end{thm}
 
\section{$\Z$-graded rings as Cuntz--Pimsner rings}

Given a ring $A$ and a subring $R\subseteq A$, it is natural to ask whether $A$ can be realised as the Cuntz--Pimsner ring of some $R$-system. If $R=A$, then the answer is trivially yes: writing $\psi$ for the zero map, it follows almost immediately that $(\{0\},\{0\},\psi)$ is an $A$-system satisfying condition (FS), $\{0\}\subseteq A$ is a uniquely maximal faithful $\psi$-compatible two-sided ideal of $A$, and $A\cong \TT_{(\{0\},\{0\},\psi)}\cong \O_{(\{0\},\{0\},\psi)}$. Here, the key to making use of the Cuntz--Pimsner ring construction is to relate structural properties of the coefficient ring $R$ to structural properties of $A$,  so we would like $R$ to be `smaller' and more tractable than $A$ itself.  In this section, we consider a $\mathbb{Z}$-graded ring $A$ and derive sufficient conditions for when  $A$ is the Cuntz--Pimsner ring of a $R$-system where $R$ is a subring of $A_0$ (Theorem~\ref{mainth0}).

Before we state and prove our main theorem, we fix some more notation. Let $R$ be a ring, $M$ a left $R$-module and $I$ a subset of $M$. The \emph{left annihilator} $$\ann_R(I)=\{r\in R: rx=0 \text{ for all } x\in I\}$$ of $I$ by $R$ is a left ideal of $R$. In the case that $I$ is a submodule of $M$, $\ann_R(I)$ is a two-sided ideal of $R$. 

\begin{thm}\label{mainth0}
Let $A=\bigoplus_{i\in \Z}A_i$ be a $\Z$-graded ring, $R$ a subring of $A_0$, and $I\subseteq A_1$ and  $J\subseteq A_{-1}$ additive subgroups such that 
\begin{enumerate}[\upshape(1)]
\item $RI, IR \subseteq I$, $RJ, JR \subseteq J$ and 
$JI\subseteq R$;

\medskip

\item For any finite subset $\{i_1,\dots,i_n\}\subseteq I$ there is an element $a$ in $IJ$ such that $a i_l=i_l$ for each $1\leq l \leq n$, and for 
any finite subset $\{j_1,\dots,j_m\}\subseteq J$ there is an element $b$ in $IJ$ such that $j_l b=j_l$ for each $1\leq l \leq m$;

\medskip

\item For $r\in \ann_R(I)^\bot$ and $a\in IJ$, if $r-a \in \ann_{A_0}(I)$, then $a\in R$
\footnote{Notice that if $IJ\subseteq R$, then Condition (3) is trivially satisfied.};

\medskip 

\item $\ann_R(I)\cap \ann_R(I)^\bot=\{0\}$.
\end{enumerate}
Then there exists an $R$-bimodule homomorphism $\psi:J\otimes_R I \rightarrow R$ such that $\psi(j\otimes_R i)=ji$ for each $j\in J, i\in I$, and $(J,I,\psi)$ is an $R$-system. Furthermore, there is a graded isomorphism from the Cuntz--Pimsner ring $\mathcal O_{(J,I,\psi)}$ of the $R$-system $(J,I,\psi)$ to the subring of $A$ generated by $R,I,J$. 
\begin{proof}
First note that from Condition (1) we see that $I$ and $J$ are $R$-bimodules and there exists an $R$-bimodule homomorphism 
$\psi:J\otimes_R I \rightarrow R$ such that $\psi(j\otimes_R i)=ji$ for each $j\in J, i\in I$. Thus, $(J,I,\psi)$ is an $R$-system. Let $S$ denote the subring of $A$ generated by $R\cup I\cup J$, and $i_R, i_I, i_J$ denote the inclusion maps from $R,I,J$ to $S$. Clearly $(i_J, i_I, i_R, S)$ is a surjective covariant representation of the system $(J,I,\psi)$. In this setting condition (FS)  translates to Condition (2). Hence, by Lemma~\ref{lemmaforpi} there exists a ring homomorphism $\pi_{i_I, i_J}:\mathcal F_{J}(I)\rightarrow S$ such that $\pi_{i_I, i_J}(\theta_{i,j})=i_I(i)i_J(j)=ij$ for each $i\in I, j\in J$. By \eqref{delta} the map $\Delta: R\rightarrow{} \End_R(I)$ is given by $\Delta(r)(x)=rx$ for $r\in R$ and $x\in I$, and thus $\ker \Delta=\ann_R(I)$.  By Condition (4) and Lemma ~\ref{lemmaforexistence}, $\Delta^{-1}(\mathcal F_{J}(I))\cap \ann_R(I)^\bot$ is a uniquely maximal faithful $\psi$-compatible two-sided ideal of $R$, and thus the Cuntz--Pimsner ring $\mathcal O_{(J,I,\psi)}$ is well-defined. We let $\iota_R^{CP}:R\rightarrow  \mathcal O_{(J,I,\psi)}$, $\iota_I^{CP}:I\rightarrow \mathcal O_{(J,I,\psi)}$, and $\iota_J^{CP}:J\rightarrow \mathcal O_{(J,I,\psi)}$ denote the maps whose images generate $\mathcal O_{(J,I,\psi)}$. We will show that this ring is graded isomorphic to $S$. 

We claim that 
\begin{equation}\label{mmnnbb}
\Delta^{-1}(\mathcal F_{J}(I))\cap \ann_R(I)^\bot=IJ \cap R. 
\end{equation} 
Clearly, $IJ \cap R \subseteq \Delta^{-1}(\mathcal F_{J}(I))$. We now show that 
\begin{equation}\label{anni}
IJ \cap R \subseteq \ann_R(I)^\bot.
\end{equation} Let $x=\sum_{l=1}^n i_lj_l \in IJ\cap R$ and $y \in \ann_R(I)$. Then 
$yx=\sum_{l=1}^n (yi_l)j_l=0$. On the other hand, $xy=\sum_{l=1}^n i_l (j_ly)$. By Condition (1) each $j_ly \in J$, and by Condition (2) there is some $b\in IJ$ such that $j_lyb=j_ly$ for each $l\in \{1,\ldots, n\}$. But $yb=0$ as $y \in \ann_R(I)$ and $b\in IJ$. Hence $xy=0$, and so $x \in \ann_R(I)^\bot$. Thus $IJ \cap R \subseteq \Delta^{-1}(\mathcal F_{J}(I))\cap \ann_R(I)^\bot$. For the reverse containment, suppose $x \in \Delta^{-1}(\mathcal F_{J}(I))\cap \ann_R(I)^\bot$. It follows that there is an element $\sum_{l=1}^ni_l j_l \in IJ$ with $i_l\in I$ and $j_l\in J$ such that $x-\sum_{l=1}^n i_l j_l \in \ann_{A_0}(I)$. 
Condition (3) guarantees that $\sum_{l=1}^n i_l j_l \in IJ \cap R$, and so $x-\sum_{l=1}^n i_l j_l \in \ann_R(I)$.  Since $x\in  \ann_R(I)^\bot$, by \eqref{anni} it follows that $x-\sum_{l=1}^n i_l j_l \in  \ann_R(I)^\bot$. Now Condition (4) implies that $x=\sum_{l=1}^ni_l j_l  \in IJ$, proving the claim.

Using \eqref{mmnnbb} it is routine to check that $\pi_{i_I, j_J}(\Delta(x))=\iota_R(x)$ for each
\[x\in \Delta^{-1}(\mathcal F_{J}(I))\cap \ann_R(I)^\bot=IJ \cap R,\] and so the surjective covariant representation $(i_J, i_I, i_R, S)$ of $(J,I,\psi)$ is Cuntz--Pimsner invariant. By \cite[Theorem 3.18]{CO11}, there exists a surjective ring homomorphism $\eta:\mathcal O_{(J,I,\psi)}\rightarrow S$ such that $\eta\circ \iota_R^{CP}=i_R$, $\eta\circ \iota_I^{CP}=i_I$, and $\eta\circ \iota_J^{CP}=i_J$. As $R\subseteq A_0$, $I\subseteq A_1$, and $J\subseteq A_{-1}$, it follows that $\eta$ is graded. Finally, since $\eta(\iota_R^{CP} (r))=r$ for all $r\in R$, the graded uniqueness theorem for Cuntz--Pimsner rings (Theorem~\ref{gut}) ensures that $\eta$ is also injective. 
\end{proof}
\end{thm}

Specialising Theorem~\ref{mainth0} to the situation where $I=A_1$, $J=A_{-1}$, and $R=A_0$, we have the following corollary. 

\begin{cor}\label{grcp1}
Let $A=\bigoplus_{i\in \Z}A_i$ be a $\Z$-graded ring such that 
\begin{enumerate}[\upshape(1)]
\item $A_n=A_1^n$ and  $A_{-n}=A_{-1}^n$ for $n>0$;

\medskip

\item For $\{a_1,\dots,a_n\}\subseteq A_1$ there is an $r \in A_1A_{-1}$ such that $r a_l=a_l$ for each $1\leq l \leq n$, and for 
$\{b_1,\dots,b_m\}\subseteq A_{-1}$ there is an $s \in A_1A_{-1}$ such that $b_l s=b_l$ for each $1\leq l \leq m$;

\medskip

\item $\ann_{A_0}(A_1)\cap \ann_{A_0}(A_1)^\bot=\{0\}$.
\end{enumerate}
Then there is a graded isomorphism from the Cuntz--Pimsner ring of the $A_0$-system $(A_{-1},A_1,\psi)$ to $A$, where $\psi:A_{-1}\otimes_{A_0}A_1 \rightarrow A_0$ is the $A_0$-bilinear map that sends $a\otimes_{A_0} b$ to $ab$ for each $a\in A_{-1}, b\in A_1$. 
\end{cor}

Recall that a $\Z$-graded ring $A=\bigoplus_{i\in \Z}A_i$ is said to be \emph{strongly graded} if $A_m A_n=A_{m+n}$ for each $m,n\in \Z$ (see \cite{hazi}). We say that a $\Z$-graded ring $A$ has \emph{graded local units} if for any finite set of homogeneous elements  $\{x_{1}, \cdots, x_{n}\}\subseteq A$, there exists a homogeneous
idempotent $e\in A$ such that $\{x_{1}, \cdots, x_{n}\}\subseteq eAe$. Equivalently, $A$
has graded local units, if $A_0$ has local units and $A_0A_{n}=A_{n}A_0=A_{n}$ for every $n \in \Z$.

It is not difficult to see that for a strongly graded ring $A$, if a graded homomorphism $\phi: A\rightarrow B$  restricted to $A_0$ is injective, then $\phi$ is injective (see \cite[Corollary 1.3.9]{no}). However our results allows us to place this in the general framework of  the Graded Uniqueness Theorem for Cuntz--Pimsner rings (Theorem~\ref{gut}). 

\begin{cor}[Uniqueness theorem for strongly graded rings]\label{grcp2}
Let $A=\bigoplus_{i\in \Z}A_i$ be a strongly $\Z$-graded ring with graded local units and $B$ a $\Z$-graded ring. If $\phi:A \rightarrow B$ is a graded 
homomorphism such that $\phi|_{A_0}$ is injective then $\phi$ is injective. 
\end{cor}
\begin{proof}

We will check conditions (1)--(3) in Corollary~\ref{grcp1}. Condition (1) follows from the definition of strongly graded rings. Since $A_1A_{-1}=A_0$ which contains local units for $A$, Condition (2) of Corollary~\ref{grcp1} is immediate. Observe that $\ann_{A_0}(A_1)=\{0\}$ and thus Condition (3) of Corollary~\ref{grcp1} follows. Thus there is a graded isomorphism from $\mathcal{O}_{(A_1, A_{-1},\psi)}$ to $A$, and the result now follows from Theorem~\ref{gut}. 
\end{proof}

Next we show how \cite[Example~5.7]{CO11}, \cite[Example~5.5]{CO11}, and \cite[Examples~1.10 and 5.8]{CO11} fit into  the framework of Theorem~\ref{mainth0} and Corollary~\ref{grcp1}.  First we consider corner skew Laurent polynomial rings, studied  in \cite{arabrucom}, where their $K_1$-groups were calculated. The construction is a special case of the so-called fractional skew monoid rings constructed in \cite{ara2004} (see also \cite[\S1.6.2]{hazi}). Corner skew Laurent polynomial rings are characterised by the following property: $A=\oplus_{n\in \mathbb Z}A_n$ is a $\mathbb Z$-graded ring and $A_1$ has a left invertible element (\cite[Lemma 2.4]{ara2004}, \cite[Theorem 1.6.9]{hazi}).  This characterisation is used in \cite[Example 1.6.14]{hazi} to show that Leavitt path algebras associated to finite graphs with no sinks are examples of such rings. Using our Corollary~\ref{grcp1}, in Example~\ref{boathouse11} we realise corner skew Laurent polynomial rings as a special case of Cuntz--Pimsner rings. In Example~\ref{crossed product}, we show how Corollary~\ref{grcp1} can be applied to crossed products by automorphisms. Finally, in Example~\ref{boathouse4} we show how Theorem~\ref{mainth0} can be applied to Leavitt path algebras of arbitrary graphs. In this last example the coefficient ring of our system is (in general) smaller than the zero graded part of the Leavitt path algebra, and so requires the full power of Theorem~\ref{mainth0} rather than just Corollary~\ref{grcp1}.

\begin{example}[Corner skew Laurent polynomial rings]
\label{boathouse11}
Let $R$ be a unital ring, $p\in R$ an idempotent, and $\alpha:R \rightarrow pRp$ a ring isomorphism. Recall from \cite{ara2004} that the \emph{corner skew Laurent polynomial ring} is the universal unital ring $R[t_{+},t_{-},\alpha]$ generated by elements $t_+$, $t_-$, and $\{\phi(r):r\in R\}$, where $\phi:R\rightarrow R[t_{+},t_{-},\alpha]$ is a unital ring homomorphism, satisfying the following relations
\begin{enumerate}
\item $t_- t_+=1_{R[t_+,t_-,\alpha]}$;
\item $t_+ t_-=\phi(p)$;
\item $\phi(r)t_-=t_-\phi(\alpha(r))$ for $r\in R$; and
\item $t_+\phi(r)=\phi(\alpha(r))t_+$ for $r\in R$.
\end{enumerate}
Observe that if $r\in R$ and $m,n\geq 1$, then 
\[
t_+^m\phi(r)t_-^n=t_+^{m-1}\phi(\alpha(r))t_+t_-^n=t_+^{m-1}\phi(\alpha(r)p)t_-^{n-1}=t_+^{m-1}\phi(\alpha(r))t_-^{n-1}
\]
and, because $t_- t_+=1_{R[t_+,t_-,\alpha]}$,
\begin{align*}
t_-^m\phi(r)t_+^n
&=t_-^{m-1}(t_-t_+)t_-\phi(r)t_+(t_-t_+)t_+^{n-1}\\
&=t_-^{m-1}t_-(t_+t_-)\phi(r)(t_+t_-)t_+t_+^{n-1}\\
&=t_-^{m-1}t_-\phi(prp)t_+^n\\
&=t_-^{m-1}\phi(\alpha^{-1}(prp))t_-t_+t_+^{n-1}\\
&=t_-^{m-1}\phi(\alpha^{-1}(prp))t_+^{n-1}.
\end{align*}
It follows that every element of $R[t_{+},t_{-},\alpha]$ can be written in the form
\[
\phi(r_n)t_+^n+\cdots +\phi(r_1)t_++\phi(r_0)+t_-\phi(r_{-1})+\cdots + t_-^m\phi(r_{-m})
\] 
for some $m,n\in \N_+$ and $r_i\in R$.
Furthermore, if we define
\[
A_n:=
\begin{cases}
t_-^n \phi(R) & \text{if $n>0$}\\
\phi(R) & \text{if $n=0$}\\
\phi(R)t_+^{-n} & \text{if $n<0$,}
\end{cases}
\]
then $\bigoplus_{n\in \Z}A_n$ is a $\Z$-grading of $R[t_{+},t_{-},\alpha]$. 

We now show that this example fits into our framework. Observe that 
\[
t_-=1_{R[t_+,t_-,\alpha]}t_-=(t_-t_+)t_-=t_-(t_+t_-)=t_-\phi(p)\in A_1
\]
and
\[
t_+=t_+1_{R[t_+,t_-,\alpha]}=t_+(t_- t_+)=(t_+t_-)t_+=\phi(p)t_+\in A_{-1}.
\]
It follows immediately that Condition (1) of Corollary~\ref{grcp1} holds. Furthermore, since
\[
1_{R[t_+,t_-,\alpha]}=t_- t_+\in A_1A_{-1}, 
\]
we see that Condition (2) of Corollary~\ref{grcp1} holds. Lastly, observe that if $\phi(r)\in \ann_{A_0}(A_1)$, then
\[
\phi(r)=\phi(r)1_{R[t_+,t_-,\alpha]}\in \phi(r)A_1A_{-1}=\{0\},
\]
and so we see that Condition (3) of Corollary~\ref{grcp1} holds.
Consequently, if $\psi: A_{-1} \otimes_{A_0} A_1 \rightarrow A_0$ is defined by $\psi(x \otimes_{A_0}y) =xy$, then there is a graded isomorphism from $R[t_{+},t_{-},\alpha]$ to the  Cuntz--Pimsner ring of the $A_0$-system $(A_{-1}, A_1,\psi)$. 
\end{example}

\begin{example}[Crossed products by automorphisms]
\label{crossed product}
By applying Corollary~\ref{grcp1}, we see that \cite[Example~5.5]{CO11}, mentioned in the introduction, fits into our framework as well.  Let $R$ be a ring with local units and $\varphi$ be an automorphism of $R$. Then the crossed product $A:= R \times_{\varphi} \mathbb{Z}$ is the universal ring generated by $\{[r,k] : r \in R \text{ and } k \in \mathbb{Z}\}$ satisfying $[r_1,k]+[r_2,k] =[r_1+r_2, k]$ and 
$[r_1,k_1][r_2,k_2] = [r_1\varphi^{k_1}(r_2),k_1+k_2]$. Then $A$ is a strongly $\mathbb{Z}$-graded ring with $A_n:= \{[r,n] : r \in R\}$ for each $n \in \mathbb{Z}$.  Consequently, it is straightforward to show that the hypotheses of Corollary~\ref{grcp1} are satisfied. Hence, $A$ is isomorphic to the Cuntz--Pimsner ring of the $A_0$-system $(A_{-1}, A_1, \psi)$ where $\psi: A_{-1} \otimes_{A_0} A_1 \to A_0$ is defined by
$\psi([r_1,-1] \otimes_{A_0} [r_2,1]) = [r_1\varphi^{-1}(r_2),0]$.
\end{example}

\begin{example}[Leavitt path algebras]
\label{boathouse4}
Let $E=(E^0, E^1, r,s)$ be a directed graph. Recall from \cite{AAS}  that the Leavitt path algebra $A:=L_K(E)$  over a unital commutative ring $K$ is generated by $\{v: v\in E^0\}\cup \{e: e\in E^1\}\cup \{e^*: e\in E^1\}$ subject to the following relations 
\begin{itemize}
\item[(1)] $uv=\delta_{u, v}u$ for $u, v\in E^0$;

\item[(2)] $r(e)e=e=es(e)$ for $e\in E^1$;

\item[(3)] $s(e)e^*=e^*=e^*r(e)$ for $e\in E^1$;

\item[(4)] $e^*f=\delta_{e, f}s(e)$ for $e, f\in E^1$;

\item[(5)] $\sum_{e\in r^{-1}(v)}ee^*=v$ if $v\in E^0$ with $r^{-1}(v)$ a finite nonempty set.
\end{itemize} Relations $(4)$ and $(5)$ are called the \emph{Cuntz--Krieger relations}. 
Note that we are using the `Southern Hemisphere' convention where a path is a finite sequence of edges $e_1 ... e_n$ such that $s(e_i) = r(e_{i+1}).$
$A$ is naturally $\Z$-graded with $|v|=0$, $|e|=1$ and $|e^*|=-1$ for $v\in E^0$ and $e\in E^1$.

We set $R:=\mathrm{span} \{v: v\in E^0\}$, $I:=\mathrm{span} \{e: e\in E^1\}$, and $J:=\mathrm{span} \{e^*: e\in E^1\}$. We apply  Theorem~\ref{mainth0} to show that there is a graded ring isomorphism from the Cuntz--Pimsner ring of the $R$-system $(J,I,\psi)$ to $A$. By relations (1)--(4) of $A$, $R$ is a subring of $L_K(E)$, $RI, IR\subseteq I$, $RJ, JR\subseteq J$, and $JI\subseteq R$. Thus Condition (1) of Theorem~\ref{mainth0} holds. Note that it is possible that $IJ$ is not contained in $R$. It is straightforward to show that $\ann_R(I)=\mathrm{span}\{v: vE^1=\emptyset\}$ and $\ann_R(I)^\bot=\mathrm{span}\{v:vE^1\neq\emptyset\}$, and so $\ann_R(I)\cap \ann_R(I)^\bot=\{0\}$. Thus, Condition (4) of Theorem~\ref{mainth0} is satisfied. 

Next we show that Condition (2) is satisfied. Let $\{i_1,\ldots, i_n\}\subseteq I$ be a finite set. For each $l\in \{1,\ldots, n\}$ there exist finite sets $F_l\subseteq E^1$ and $\{\lambda_e:e\in F_l\}\subseteq K$ such that $i_l=\sum_{e\in F_l} \lambda_e e$. Since $ee^*f=\delta_{e,f}f$ for any $e,f\in E^1$, if we set $a:=\sum_{e\in \cup_{l=1}^n F_l}ee^*\in IJ$, then $ai_l=i_l$ for each $l\in \{1,\ldots, n\}$. Similarly, if $\{j_1,\ldots, j_m\}\subseteq J$ is a finite set, we may write $j_l=\sum_{e\in F_l} \lambda_e e^*$ where $X_l\subseteq E^1$ and $\{\lambda_e:e\in X_l\}\subseteq K$ are finite sets. With $b:=\sum_{e\in \cup_{l=1}^m X_l}ee^*\in IJ$, we then have that $j_lb=j_l$ for each $l\in \{1,\ldots, m\}$.

It remains to check Condition (3) of Theorem~\ref{mainth0}. Let $r:=\sum_{v\in F}\lambda_v v\in \ann_R(I)^\bot$, $a:=\sum_{(e,f)\in G}\mu_{(e,f)}ef^*\in IJ$, where $F\subseteq E^0$, $G\subseteq E^1\times E^1$, $\{\lambda_v:v\in F\}$ and $\{\mu_{(e,f)}:(e,f)\in G\}\subseteq K$ are finite sets.  Suppose that $r-a\in \ann_{A_0}(I)$. Observe that for any $u,w\in E^0$,
\begin{equation}
\label{reducing to single vertex}
\delta_{u,w}\lambda_u u-\sum_{\substack{(e,f)\in G\\ r(e)=u,r(f)=w}}\mu_{(e,f)}ef^*=u(r-a)w\in \ann_{A_0}(I),
\end{equation} since $RI\subseteq I$ (where $\lambda_u:=0$ if $u\notin F$). Thus,
\[
\sum_{\substack{(e,f)\in G\\ r(e)\neq r(f)}}\mu_{(e,f)}ef^*, 
\sum_{\substack{(e,f)\in G\\ r(e)=r(f)\not\in F}}\mu_{(e,f)}ef^*
\in\ann_{A_0}(I).
\]
We claim that 
\begin{equation}
\label{distinct ranges}
\sum_{\substack{(e,f)\in G\\ r(e)\neq r(f)}}\mu_{(e,f)}ef^*=0
\end{equation}
and
\begin{equation}
\label{common range not in F}
\sum_{\substack{(e,f)\in G\\ r(e)=r(f)\not\in F}}\mu_{(e,f)}ef^*=0.
\end{equation}
For $k\in E^1\subseteq I$, we have
\[
0=\bigg(\sum_{\substack{(e,f)\in G\\ r(e)\neq r(f)}}\mu_{(e,f)}ef^*\bigg)k=\sum_{\substack{(e,k)\in G\\ r(e)\neq r(k)}}\mu_{(e,k)}e.
\]
Since the collection of edges $\{e:e\in E^1\}$ is a linearly independent set in $A$, this forces $\mu_{(e,k)}=0$ whenever $(e,k)\in G$ and $r(e)\neq r(k)$. This gives \eqref{distinct ranges}. Similar calculations show that \eqref{common range not in F} holds as well. Now suppose that $v\in F$ and $\lambda_v\neq 0$. By \eqref{reducing to single vertex}, 
\[
\lambda_v v-\sum_{\substack{(e,f)\in G\\ r(e)=r(f)=v}}\mu_{(e,f)}ef^*\in \ann_{A_0}(I).
\] 
We claim that  
\begin{enumerate}
\item[(i)] if $(e,f)\in G$ and $r(e)=r(f)=v$, then $\mu_{(e,f)}=\lambda_v$ when $e=f$ and $\mu_{(e, f)}$ is zero otherwise; and
\item[(ii)] if $e\in vE^1$ then $(e,e)\in G$ (and so in particular, $vE^1$ is a finite set). 
\end{enumerate}
For $k\in vE^1$, we have
\[
0=\bigg(\lambda_v v-\sum_{\substack{(e,f)\in G\\ r(e)=r(f)=v}}\mu_{(e,f)}ef^*\bigg)k=\lambda_vk-\sum_{\substack{(e,k)\in G\\ r(e)=v}}\mu_{(e,k)}e.
\]
Again since the edges $\{e: e\in E^1\}$ are linearly independent, we conclude (since $\lambda_v\neq 0$) that $(k,k)\in G$, $\lambda_v=\mu_{(k,k)}$, and $\mu_{(e,k)}=0$ whenever $e\in vE^1\setminus \{k\}$ and $(e,k)\in G$. Thus (i) and (ii) hold. 
Since $r\in \ann_R(I)^\bot$, we also know that $vE^1\neq \emptyset$, and so by relation (5) of $A$, we have
\begin{equation}
\label{applying the CK relation}
\sum_{\substack{(e,f)\in G\\ r(e)=r(f)=v}}\mu_{(e,f)}ef^*=\sum_{\substack{(e,e)\in G\\ r(e)=v}}\lambda_v ee^*=\sum_{e\in vE^1}\lambda_v ee^*=\lambda_v v.
\end{equation}
Therefore, 
\begin{align*}
a=\sum_{(e,f)\in G}\mu_{(e,f)}ef^*
&=\sum_{\substack{(e,f)\in G\\r(e)\neq r(f)}}\mu_{(e,f)}ef^*
+\sum_{\substack{(e,f)\in G\\r(e)=r(f)\not\in F}}\mu_{(e,f)}ef^*
\sum_{\substack{(e,f)\in G\\r(e)=r(f)\in F}}\mu_{(e,f)}ef^*\\
&=\sum_{v\in F}\sum_{\substack{(e,f)\in G\\ r(e)=r(f)=v}}\mu_{(e,f)}ef^*\\
&=\sum_{v\in F}\lambda_v v\in R, 
\end{align*}
where the penultimate equality comes from \eqref{distinct ranges} and \eqref{common range not in F}, and the final equality from \eqref{applying the CK relation}. Hence, we conclude that Condition (3) holds.
\end{example}

\section{Application to Steinberg algebras}
\label{Steinberg algebras}
Steinberg algebras were introduced in~\cite{St} in the context of discrete inverse semigroup algebras and independently in \cite{CFST} as a model for Leavitt path algebras.  Throughout this section, $K$ denotes an arbitrary field.
Let $\mG$ be an ample Hausdorff groupoid, that is, 
a topological groupoid whose topology is Hausdorff and has a base of compact open bisections. The Steinberg $K$-algebra of $\mG$ is the $K$-linear span of characteristic functions
$1_B:\mG \to K$ with $B$ a compact open bisection of $\mG$; see \cite[Lemma 3.3]{CFST}. Addition and scalar multiplication are pointwise and multiplication is given by convolution (which reduces to 
$1_B1_D = 1_{BD}$ for compact open bisections $B$ and $D$). For the basics on Steinberg algebras see \cite{CFST,St}.

We are interested in $\Z$-graded Steinberg algebras where the grading comes 
from a continuous cocycle $c$, that is,
a continuous homomorphism $c:\mG \to \mathbb{Z}$ (where $\mathbb{Z}$ has the discrete topology).  
The homogeneous components are given by 
\[A_K(\mG)_n := \{f \in A_K(\mG) : f(\gamma) \neq 0 \Longrightarrow c(\gamma)=n\}.\]

For any clopen set $H \subseteq \mG$, with some abuse of notation, we write
\[A_K(H) : = \{f \in A_K(\mG) : f(\gamma) = 0 \text{ for } \gamma \notin H\}.\]
Thus $A_K(H)$ consists of functions that can be written as a linear combination of characteristic functions associated to compact open bisections contained in $H$. With this convention, for each $n\in \Z$, let $\mathcal{G}_n:=c^{-1}(n)$ which is clopen subset of $\mG$ as $c$ is continuous and $\mathbb{Z}$ is discrete.  Then we have $A_K(\mG_n) = A_K(\mG)_n$ (see \cite{CFST,CE-M,CS2015} for more details).

With some moderate hypotheses, it is not hard to show how Corollary~\ref{grcp1} can be applied to Steinberg algebras.  This boils down to applying Theorem~\ref{mainth0} with $I := A_K(\mG)_1$, $J:=A_K(\mG)_{-1}$, and $R := A_K(\mG)_0$, giving us an algebraic analogue of \cite[Proposition~10]{RRS17} (see Corollary~\ref{cor:st}). This is the natural choice and we expect suffices in most situations.  
However, we can do things more generally and use smaller $I,J,R$, as the next Theorem shows. 

\begin{thm}
\label{Steinbergalgebras}
Let $\mathcal{G}$ be a locally compact Hausdorff ample groupoid and $c:\mathcal{G}\rightarrow \Z$ be a continuous cocycle.  Suppose we have clopen sets $H_0\subseteq \mathcal{G}_0$, $H_1\subseteq \mathcal{G}_1$, and $H_{-1}\subseteq \mathcal{G}_{-1}$. Define $R:=A_K(H_0)\subseteq A_K(\mathcal{G})_0$, $I:=A_K(H_1)\subseteq A_K(\mathcal{G})_1$, and $J:=A_K(H_{-1}) \subseteq A_K(\mathcal{G})_{-1}$.  

\begin{enumerate}[\upshape(i)]
\item
If $H_0$ is closed under multiplication, $H_0 H_1\cup H_1 H_0\subseteq H_1$, and $H_0 H_{-1}\cup H_{-1}H_0\subseteq H_{-1}$\footnote{This condition follows from the previous one if $H_{-1}=H_1^{-1}$ and $H_0$ is closed under taking inverses.}, $H_{-1}H_1\subseteq H_0$, and $r(H_1)\cup s(H_{-1})\subseteq H_1 H_{-1}$, then $R$ is a subring of $A_K(\mathcal{G})_0$ and Conditions (1), (2), and (3) of Theorem~\ref{mainth0} are satisfied. 

\medskip

\item
If $H_0$ and $H_1$ have the property that 
\begin{equation}
\label{hypothesis}
B\subseteq H_0 \text{ is a compact open bisection and } s(B)\cap r(H_1)=\emptyset \Longrightarrow s(B)\subseteq H_0,
\footnote{This condition is automatic if either $\mathcal{G}^{(0)}\subseteq H_0$, or if $H_0$ is closed under taking inverses and under multiplication (or in other words $H_0$ is a subgroupoid of $\mathcal{G}_0$).}\end{equation} then Condition (4) of Theorem~\ref{mainth0} is satisfied. 

\medskip

\item
If every element of $\mathcal{G}$ can be written as the product of elements from $H_0$, $H_1$, and $H_{-1}$, then $R,I,J$ generate $A_K(\mathcal{G})$ as a ring.
\end{enumerate}
If $H_0$, $H_1$, and $H_{-1}$ satisfy Conditions \rm{(i)}--\rm{(iii)}, then there is a graded algebra isomorphism from the Steinberg algebra $A_K(\mathcal{G})$ to the Cuntz--Pimsner ring of the $A_K(H_0)$-system $(A_K(H_1), A_K(H_{-1}),\psi)$. 
\end{thm}

Before we give the proof of this theorem, we need the following lemma.

\begin{lemma}
\label{generating lemma}
Let $\mathcal{G}$ be a locally compact Hausdorff ample groupoid. If $C\subseteq \mathcal{G}$ is a compact open bisection and there exist clopen sets $D_1,\ldots,D_n\subseteq \mathcal{G}$ such that every element of $C$ can be written as the product of elements from $\cup_{i=1}^n D_i$, then $1_C$ belongs to the subring of $A_K(\mathcal{G})$ generated by $A_K(D_1),\ldots, A_K(D_n)$.
\begin{proof}
Let $C\subseteq \mathcal{G}$ be a compact open bisection. For $\gamma\in C$ write $\gamma=\mu_1\ldots \mu_{n_\gamma}$ where each $\mu_i\in D^i\in \{D_1,\ldots, D_n\}.$ Suppose that $n_\gamma>1$. For each $i\in\{1,\ldots, n_\gamma-1\}$ choose a compact open bisection $B_i^\gamma\subseteq D^i$ with $\mu_i\in B_i^\gamma$. Then $\mu_{n_\gamma}=\mu_{n_\gamma-1}^{-1}\ldots \mu_1^{-1}\gamma\in \big((B_{n_\gamma-1}^\gamma)^{-1}\ldots (B_{1}^\gamma)^{-1}C\big)\cap D^{n_\gamma}$, which is an open set because multiplication and inversion in $\mathcal{G}$ are open maps. Hence, we can choose a compact open bisection $B_{n_\gamma}^\gamma\subseteq \big((B_{n_\gamma-1}^\gamma)^{-1}\ldots (B_{1}^\gamma)^{-1}C\big)\cap D^{n_\gamma}$ containing $\mu_{n_\gamma}$. Thus, $\gamma=\mu_1\ldots \mu_{n_\gamma}\in B_1^\gamma \ldots B_{n_\gamma}^\gamma$ and $B_1^\gamma\ldots B_{n_\gamma}^\gamma\subseteq B_1^\gamma\ldots B_{n_\gamma-1}^\gamma(B_{n_\gamma-1}^\gamma)^{-1}\ldots (B_{1}^\gamma)^{-1}C\subseteq C$. (If $n_\gamma=1$, then $\gamma=\mu_1\in C\cap D^1$, and we can choose a compact open bisection $B_1^\gamma\subseteq C\cap D^1$ with $\gamma\in B_1^\gamma$.) Consequently, $\{B_1^\gamma\ldots B_{n_\gamma}^\gamma:\gamma\in C\}$ is an open cover for $C$ with $B_1^\gamma\ldots B_{n_\gamma}^\gamma\subseteq C$ for each $\gamma\in C$. Since $C$ is compact, it follows that there exist $\gamma_1,\ldots, \gamma_k\in C$ such that $C=\bigcup_{j=1}^k B_1^{\gamma_j}\ldots B_{n_{\gamma_j}}^{\gamma_j}$. If for each $j\in \{1,\ldots, k\}$ we define
\[
\widetilde{B}_{n_{\gamma_j}}^{\gamma_j}
:=
\begin{cases}
B_{n_{\gamma_j}}^{\gamma_j}\setminus 
\bigg(
\big(B_1^{\gamma_j}\ldots B_{n_{\gamma_j}-1}^{\gamma_j}\big)^{-1}
\Big(\bigcup_{i=j+1}^k B_1^{\gamma_i}\ldots B_{n_{\gamma_i}}^{\gamma_i}\Big)
\bigg)
& \text{if $n_{\gamma_j}>1$}\\
B_1^{\gamma_j}\setminus \Big(\bigcup_{i=j+1}^k B_1^{\gamma_i}\ldots B_{n_{\gamma_i}}^{\gamma_i}\Big)
&\text{if $n_{\gamma_j}=1$}
\end{cases},
\]
then $C=\bigcup_{j=1}^k B_1^{\gamma_j}\ldots B_{n_{\gamma_j}-1}^{\gamma_j}\widetilde{B}_{n_{\gamma_j}}^{\gamma_j}$ and this union is disjoint. Thus,
\[
1_C=\sum_{j=1}^k 1_{B_1^{\gamma_j}}\ldots 1_{B_{n_{\gamma_j}-1}^{\gamma_j}}1_{\widetilde{B}_{n_{\gamma_j}}^{\gamma_j}}.
\]
Since $B_i^{\gamma_j}$ is contained in one of $D_1,\ldots,D_n$, for each $j\in \{1,\ldots,k\}$ and $i\in \{1,\ldots, n_{\gamma_j}\}$, we see that $1_C$ is contained in the ring generated by $A_K(D_1),\ldots, A_K(D_n)$.
\end{proof}
\end{lemma}

\begin{rmk}
\label{fixed order}
Observe that in the proof of Lemma~\ref{generating lemma} we showed that if $C\subseteq D_1\cdots D_n$, then $1_C\in A_K(D_1)\cdots A_K(D_n)$.
\end{rmk}

\begin{proof}[Proof of Theorem~\ref{Steinbergalgebras}]
It is routine to check that if $H_0$ is closed under multiplication, then $R$ is a closed under multiplication. Clearly, $R$, $I$, and $J$ are closed under addition, and so $R$ is a subring of $A_K(\mathcal{G})_0$ and $I$ and $J$ are (additive) subgroups of $A_K(\mathcal{G})_1$ and $A_K(\mathcal{G})_{-1}$ respectively. It is routine to check that if $H_0 H_1\cup H_1 H_0\subseteq H_1$, $H_0 H_{-1}\cup H_{-1}H_0\subseteq H_{-1}$, and $H_{-1}H_1\subseteq H_0$, then $RI,IR\subseteq I$, $RJ,JR\subseteq J$, and $JI\subseteq R$. Thus, Condition (1) of Theorem~\ref{mainth0} is satisfied. 

Now suppose that $r(H_1)\cup s(H_{-1})\subseteq H_1 H_{-1}$. We will show that Condition (2) of Theorem~\ref{mainth0} is satisfied. Let $\{f_1,\ldots, f_n\}\subseteq A(H_1)$ be a finite set. For each $i\in \{1,\ldots, n\}$ choose compact open bisections $B^i_1,\ldots, B_{m_{i}}^i\subseteq H_1$ and scalars $\alpha^i_1,\ldots, \alpha^i_{m_i}\in K$ such that $f_i=\sum_{l=1}^{m_i}\alpha^i_l1_{B^i_l}$. Now choose compact open bisections $\{C_{l_i}^i\subseteq H_1:i\in \{1,\ldots, n\}, l_i\in \{1,\ldots, m_i\}\}$ such that the sets $r(C_{l_i}^i)$ where $i\in \{1,\ldots, n\}$ and $l_i\in \{1,\ldots, m_i\}$ are mutually disjoint and $\bigcup_{i=1}^n\bigcup_{l_i=1}^{m_i}r(B_{l_i}^i)= \bigcup_{i=1}^n\bigcup_{l_i=1}^{m_i}r(C_{l_i}^i)$. Note: one way to do this is to set 
\[
C_{l_i}^i:=B_{l_i}^i\setminus r^{-1}\Bigg(\bigcup_{s=i+1}^n \bigcup_{t=1}^{m_s} r(B_t^s)\cup\bigcup_{p=l_i+1}^{m_i}r(B^i_p) \Bigg)
\]
for each $i\in \{1,\ldots, n\}$ and $l_i\in \{1,\ldots, m_i\}$. We then define 
\[
r:=\sum_{i=1}^n\sum_{l_i=1}^{m_i}1_{r(C_{l_i}^i)}. 
\]
Since $r(C_{l_i}^i)\subseteq r(H_1)\subseteq H_1 H_{-1}$ for each $i\in \{1,\ldots, n\}$ and $l_i\in \{1,\ldots, m_i\}$, Lemma~\ref{generating lemma} (and Remark~\ref{fixed order}), tells us that $r\in IJ$. Hence for any $k\in \{1,\ldots, n\}$,
\begin{align*}
rf_k
=\sum_{i=1}^n\sum_{l_i=1}^{m_i} 1_{r(C_{l_i}^i)}\Big(\sum_{j=1}^{m_k}\alpha^k_j1_{B^k_j}\Big)
=\sum_{i=1}^n\sum_{l_i=1}^{m_i}\sum_{j=1}^{m_k} \alpha^k_j1_{r(C_{l_i}^i)B^k_j}
=\sum_{j=1}^{m_k} \alpha^k_j1_{B^k_j}
=f_k,
\end{align*}
where the third equality follows from the fact that 
\[r(B^k_j)\subseteq \bigcup_{i=1}^n\bigcup_{l_i=1}^{m_i}r(B_{l_i}^i)=\bigcup_{i=1}^n\bigcup_{l_i=1}^{m_i}r(C_{l_i}^i)\] for each $j\in \{1,\ldots, m_k\}$, and this last union is, by construction, disjoint. The proof that the second part of Condition (2) is satisfied is almost exactly the same, but we use sources instead of ranges. 

Now we check that Condition (3) of Theorem~\ref{mainth0} is satisfied. Let $r\in R$ and $a\in IJ$ be such that $r-a\in \ann_{A_0}(I)$. Write $r=\sum_{i=1}^m\alpha_i 1_{B_i}$ and $a=\sum_{j=1}^n\beta_j 1_{C_j}$ where $\alpha_1,\ldots, \alpha_m,\beta_1,\ldots,\beta_n\in K\setminus\{0\}$ and $B_1,\ldots,B_m\subseteq H_0$ are compact open bisections and $C_1,\ldots, C_n\subseteq H_1H_{-1}$ are disjoint compact open bisections. Looking for a contradiction, suppose that there exists $l\in \{1,\ldots, m\}$ such that $C_l\not\subseteq \cup_{i=1}^n B_i$. Choose $x\in C_l\setminus \cup_{i=1}^n B_i$ and write $x=yz$ where $y\in H_1$ and $z\in H_{-1}$. Since $s(x)=s(z)\in s(H_{-1})$, if we assume that $r(H_1)\cup s(H_{-1})\subseteq H_1 H_{-1}$, then we can choose $\xi\in H_1$ and $\sigma\in H_{-1}$ such that $s(x)=\xi\sigma$. Thus, $s(x)=r(\xi)$. Let $D\subseteq H_1$ be a compact open bisection containing $\xi$. Then $1_D\in I$ and so 
\[
0=(r-a)1_D=\sum_{i=1}^m\alpha_i 1_{B_iD}-\sum_{j=1}^n\beta_j 1_{C_jD}.
\]
By construction $x\xi\in C_lD$. If $x\xi\in C_kD$ for some $k\in \{1,\ldots, n\}\setminus \{l\}$, say $x\xi=\mu\nu$ where $\mu\in C_k$ and $\nu\in D$, then $s(\nu)=s(\xi)$ forces $\nu=\xi$ because $D$ is a bisection, and so $\mu=x$ which is impossible because $C_l\cap C_k=\emptyset$. Similarly, if $x\xi\in B_iD$ for some $i\in \{1,\ldots, m\}$, say $x\xi=\eta\tau$ where $\eta\in B_i$ and $\tau\in D$, then $s(\tau)=s(\xi)$ forces $\tau=\xi$, and so $x=\eta$ which is impossible because $x\in C_l\setminus \cup_{i=1}^n B_i$. Thus, 
\[
0=\Big(\sum_{i=1}^m\alpha_i 1_{B_iD}-\sum_{j=1}^n\beta_j 1_{C_jD}\Big)(x\xi)=\beta_l\neq 0,
\]
which is obviously not possible. Hence, $C_j\subseteq \cup_{i=1}^n B_i\subseteq H_0$ for each $j\in \{1,\ldots,m\}$, and so $a=\sum_{j=1}^n\beta_j 1_{C_j}\in A_K(H_0)=R$ as required. We have now shown that part (i) of the theorem holds. 

Before we check that part (ii) of the theorem holds, we get a handle on $\ann_R(I)$. We claim that
\begin{equation}
\label{calculating the annihilator}
\ann_R(I)=\mathrm{span}_K\{1_B: B\subseteq H_0 \text{ is a compact open bisection and } s(B)\cap r(H_1)=\emptyset\}. 
\end{equation}
Firstly, suppose that $B\subseteq H_0$ is a compact open bisection and $s(B)\cap r(H_1)=\emptyset$. Clearly, $1_B\in R$. Let $f\in I$, say $f=\sum_{i=1}^n \alpha_i 1_{D_i}$ where $\alpha_1,\ldots, \alpha_n\in K$ are some scalars and $D_1,\ldots, D_n\subseteq H_1$ are compact open bisections. Since $r(D_i)\cap s(B)\subseteq r(H_1)\cap s(B)=\emptyset$ for each $i\in \{1,\ldots,n\}$, we have
\[
1_Bf=\sum_{i=1}^n\alpha_i1_{BD_i}=0.
\]
Thus, $1_B\in \ann_R(I)$, and we conclude that 
\[
\mathrm{span}_K\{1_B: B\subseteq H_0 \text{ is a compact open bisection and } s(B)\cap r(H_1)=\emptyset\}\subseteq \ann_R(I).
\] 
For the reverse containment, let $f:=\sum_{i=1}^n \alpha_i1_{D_i}\in R$ (where $\alpha_1,\ldots, \alpha_n\in K\setminus\{0\}$ are scalars and $D_1,\ldots, D_n\subseteq H_0$ are mutually disjoint compact open bisections), and suppose that there exists $j\in \{1,\ldots,n\}$ such that $s(D_j)\cap r(H_1)\neq \emptyset$. Choose $x\in H_1$ and $y\in D_j$ such that $r(x)=s(y)$ and let $E\subseteq H_1$ be a compact open bisection containing $x$ (we can certainly do this --- if we pick a compact open bisection in $\mathcal{G}$ containing $x$ then its intersection with $H_1$ is still compact and open because $H_1$ is clopen). Looking for a contradiction, suppose there exists $k\in \{1,\ldots, n\}\setminus \{j\}$ such that $yx\in D_kE$. Thus, $yx=wz$ for some $w\in D_k$ and $z\in E$. Since $s(x)=s(z)$ and $x,z\in E$ which is a bisection, we must have $x=z$. Thus, $y=w$, which is impossible because $D_j$ and $D_k$ are disjoint. Hence,
\[
(f1_E)(yx)=\sum_{i=1}^n \alpha_i 1_{D_iE}(yx)=\alpha_j\neq 0.
\]
Thus, $f1_E\neq 0$, and so $f\not\in \ann_R(I)$. This completes the proof that \eqref{calculating the annihilator} holds. 

We now show that Condition (4) holds provided we have \eqref{hypothesis}. Suppose $f\in \ann_R(I)\cap \ann_R(I)^\bot$. By \eqref{calculating the annihilator}, $f=\sum_{i=1}^n \alpha_i 1_{B_i}$ for some choice of scalars $\alpha_1,\ldots, \alpha_n\in K$ and compact open bisections $B_1,\ldots, B_n\subseteq H_0$ with $s(B_i)\cap r(H_1)=\emptyset$ for each $i\in \{1,\ldots, n\}$. Let $D:=\bigcup_{i=1}^n s(B_i)$. Then $D$ is compact and open (because $s$ is a local homeomorphism, it is a continuous open map) and a bisection (because it is a subset of the unit space). Moreover, $s(D)=D=\bigcup_{i=1}^n s(B_i)$ is disjoint from $r(H_1)$, and so is contained in $H_0$ by \eqref{hypothesis}. Hence, by \eqref{calculating the annihilator}, $1_D\in  \ann_R(I)$, and we must have
\[
0=f1_D=\sum_{i=1}^n \alpha_i 1_{B_iD}=\sum_{i=1}^n \alpha_i 1_{B_i}=f.
\]
Thus, $\ann_R(I)\cap \ann_R(I)^\bot=\{0\}$ as required. Thus, part (ii) of the theorem holds. 

Finally, another application of Lemma~\ref{generating lemma} shows that if every element of $\mathcal{G}$ can be written as the product of elements from $H_0$, $H_1$, and $H_{-1}$, then $R=A_K(H_0)$, $I=A_K(H_0)$, and $J=A_K(H_{-1})$ generate $A_K(\mathcal{G})$ as a ring. This shows that part (iii) of the theorem holds.
\end{proof}

\begin{rmk}
It is not immediately obvious whether there exists a groupoid $\mathcal{G}$ satisfying the hypotheses of Theorem~\ref{Steinbergalgebras} with clopen sets $H_0\subseteq \mathcal{G}_0$, $H_1\subseteq \mathcal{G}_1$, $H_{-1}\subseteq \mathcal{G}_{-1}$ satisfying conditions (i)--(iii) of the same Theorem in which $H_{-1}\neq (H_1)^{-1}$. We can show that if $H_0$ is also closed under taking inverses (i.e. it is a subgroupoid), then $H_{-1}= (H_1)^{-1}$. Let $\gamma\in H_{-1}$. Since $s(H_{-1})\subseteq H_1 H_{-1}$, we have that $s(\gamma)=\mu\nu$ for some $\mu\in H_1$ and $\nu\in H_{-1}$. Thus, $\mu\nu=s(\gamma)=s(\nu)=\nu^{-1}\nu$, which forces $\mu=\nu^{-1}$ (and so $\mu^{-1}=\nu\in (H_1)^{-1}\cap H_{-1}$). Hence, $\gamma=\gamma s(\gamma)=\gamma \mu \mu^{-1}\in H_{-1}H_1(H_1)^{-1}\subseteq H_0 (H_1)^{-1}=(H_1(H_0)^{-1})^{-1}\subseteq (H_1)^{-1}$. Thus, $H_{-1}\subseteq (H_1)^{-1}$. Similar working using the fact that $r(H_1)\subseteq H_1 H_{-1}$ shows that $H_1\subseteq (H_{-1})^{-1}$. Hence, $H_{-1}=(H_1)^{-1}$. However, as the next example shows, if $H_0$ is not closed under taking inverses, then $H_{-1}$ need not equal $(H_1)^{-1}$.
\end{rmk}

\begin{example}
Let $E$ be the directed graph with vertex set $E^0=\{u,v,w\}$, edge set $E^1=\{e,f,g\}$, and range and source maps determined by $r(e)=u$, $s(e)=r(f)=r(g)=w$, and $s(f)=s(g)=v$.  
Then the boundary path groupoid $\mathcal{G}_E$ is a locally compact Hausdorff ample groupoid.    See for example \cite[Example~2.1]{CS2015} for the details.  For this graph, the usual topology on $\mathcal{G}_E$ is discrete.  The map $c:\mathcal{G}_E\rightarrow \Z$ given by $c(x,m,y)=m$ for all $(x,m,y)\in \mathcal{G}_E$ is a continuous cocycle. Hence, 
\begin{align*}
H_0 & := \{(g,0,g), (eg,0,eg),(v,0,v),(f,0,f),(f,0,g)\},\\ 
H_1 & := \{(ef,1,f),(ef,1,g),(eg,1,g),(eg,1,f),(f,1,v),(g,1,v)\},\\
H_{-1} & :=\{(f,-1,ef),(f,-1,eg),(v,-1,f),(v,-1,g)\}
\end{align*}
are clopen subsets of $(\mathcal{G}_E)_0$, $(\mathcal{G}_E)_1$, and $(\mathcal{G}_E)_{-1}$ respectively. It is then straightforward to check that $H_0$ is closed under multiplication, $H_0 H_1\cup H_1 H_0\subseteq H_1$, $H_0 H_{-1}\cup H_{-1}H_0\subseteq H_{-1}$, $H_{-1}H_1\subseteq H_0$, and $r(H_1)\cup s(H_{-1})\subseteq H_1 H_{-1}$. Furthermore, $H_0$ and $H_1$ satisfy Condition~\ref{hypothesis} (in fact $s(H_0)\subseteq H_0$), and every element of $\mathcal{G}_E$ can be written as the product of elements from $H_0\cup H_1\cup H_{-1}$. However, since $(g,-1,ef)=(ef,1,g)^{-1}$ (and $(g,-1,eg)=(eg,1,g)^{-1}$) is not in $H_{-1}$, we see that $(H_1)^{-1}\not\subseteq H_{-1}$. 
\end{example}

Finally we use Theorem~\ref{Steinbergalgebras} to show that the Steinberg algebra associated to an \emph{unperforated} $\Z$-graded groupoid $\mathcal{G}$  can be realised as the Cuntz--Pimsner ring of an $A_K(\mathcal{G}_0)$-system. This is an algebraic analogue of \cite[Proposition~10]{RRS17}, which shows that under similar hypotheses, the reduced groupoid $C^*$-algebra $C_r^*(\mathcal{G})$ associated to an (\'{e}tale) groupoid can be realised as the Cuntz--Pimsner algebra of a $C^*$-correspondence over $C_r^*(\mathcal{G}_0)$.

\begin{cor}
\label{cor:st}
Let $\mathcal{G}$ be a locally compact Hausdorff ample groupoid and $c:\mathcal{G}\rightarrow \Z$ be a continuous cocycle. Suppose that $c$ is unperforated in the sense that if $n>0$ and $g\in \mathcal{G}_n$, then there exist $g_1,\ldots, g_n\in \mathcal{G}_1$ such that $g=g_1\cdots g_n$. Then with $H_0:=\mathcal{G}_0$, $H_1:=\mathcal{G}_1$, and $H_{-1}=\mG_{-1}$ the conditions of  Theorem~\ref{Steinbergalgebras} are satisfied. Consequently, there is a graded algebra isomorphism from $A_K(\mathcal{G})$ to the Cuntz--Pimsner ring of the $A_K(\mathcal{G}_0)$-system $(A_K(\mathcal{G})_{-1}, A_K(\mathcal{G})_1,\psi)$. 
\end{cor}

\end{document}